\documentclass[11pt,a4paper]{amsart}

\usepackage[T1]{fontenc}
\usepackage[utf8]{inputenc}
\usepackage{lmodern}
\usepackage{microtype}
\usepackage{mathtools}
\usepackage{amssymb}
\usepackage{mathrsfs}
\usepackage[margin=2.6cm]{geometry}
\usepackage[hidelinks]{hyperref}
\DeclareMathOperator{\Ric}{Ric}
\DeclareMathOperator{\Cut}{Cut}

\newcommand{\dv}{\operatorname{div}}
\newcommand{\nablasq}{\nabla^{2}}
\newcommand{\pLap}{\Delta_{p}}
\newcommand{\bdry}{\partial M}
\newcommand{\tube}[1]{B_{#1}(\bdry)}
\newcommand{\ctube}[1]{\overline{B}_{#1}(\bdry)}
\newcommand{\Mreg}{M\setminus\bigl(\bdry\cup\Cut(\bdry)\bigr)}
\newcommand{\innerN}{N}
\newcommand{\outnu}{\nu}

\theoremstyle{plain}
\newtheorem{theorem}{Theorem}[section]
\newtheorem{lemma}[theorem]{Lemma}
\newtheorem{proposition}[theorem]{Proposition}
\newtheorem{corollary}[theorem]{Corollary}

\theoremstyle{definition}

\theoremstyle{remark}

\numberwithin{equation}{section}

\title[Yau-type estimates for p-harmonic functions]
{Yau-type gradient estimates for P-harmonic functions on Riemannian manifolds
with boundary under Dirichlet boundary condition}

\begin{document}
\author{Bingqi Liu}
\address{School of Mathematics and Statistics, Wuhan University, Wuhan 430072, P.R.China}
\email{2021302191907@whu.edu.cn}

\vspace{-16pt}
\begin{abstract}
A Yau-type gradient estimate is proved for positive $p$-harmonic functions on complete
Riemannian manifolds with compact boundary, under a Ricci lower bound on the Riemannian manifold and a mean
curvature lower bound on the boundary, assuming the Dirichlet condition and a sign condition on the outward normal derivative. The result extends the estimate for harmonic functions by Kunikawa and
Sakurai to the full range of $p$-Laplace operators and yields a Liouville theorem when
the curvature hypotheses are nonnegative. The cutoff method of the linear theory does
not extend when the exponent differs from two, since it requires a directional Hessian
of the distance to the boundary that Laplacian comparison cannot control. That cutoff
is replaced here by a radial barrier at the boundary and an intrinsic maximum principle
in the interior.
\end{abstract}

\maketitle

\maketitle



\vspace{0.8em}

\section{Introduction}\label{sec:intro}

Gradient estimate is a very important technique in geometric analysis and has attracted much attentions since Yau's seminal paper (\cite{Yau1975}). Yau proved the following gradient estimate for positive harmonic functions:

\begin{theorem}(\cite{Yau1975})
Let $(M,g)$ be an $n$-dimensional, complete Riemmanian manifold (without boundary). For $K\geq 0$, we assume that $Ric_M\geq -(n-1)K$. For $x_0\in M$, let $u:B_R(x_0)\to (0,\infty)$ be a positive harmonic function. Then we have
\begin{align*}
    \frac{|\nabla u|}{u}\leq C_n\left(\frac{1}{R}+\sqrt{K}\right),
\end{align*}
on $B_{\frac{R}{2}}(x_0)$, where $C_n$ is a positive constant depending only on $n$.
\end{theorem}

A consequence of the gradient estimate is the well-known Louville theorem, which states that any positive harmonic function on a complete Riemannian manifold with nonnegative Ricci curvature is a constant.
Cheng and Yau \cite{ChengYau1975} established the corresponding
local estimate on geodesic balls. For $p$-harmonic functions the picture is more
delicate, because the operator is degenerate or singular wherever the gradient
vanishes. Kotschwar and Ni \cite{KotschwarNi2009} obtained local estimates under a
lower bound on sectional curvature by a Cheng--Yau cutoff, which involves the Hessian
of the distance function. Wang and Zhang \cite{WangZhang2011} later proved the
analogue under a Ricci lower bound, using a Moser iteration that differentiates the
distance only once.

On manifolds with boundary, Laplacian comparison for the distance to the boundary
requires an additional lower bound on the mean curvature of the boundary
\cite{Kasue1982}. Under these hypotheses, Kunikawa and Sakurai \cite{KunikawaSakurai2022}
established the following Yau-type estimate for positive harmonic functions satisfying the
Dirichlet condition (i.e., it is constant on the boundary), together with a sign condition on the outward normal derivative:

\begin{theorem}\label{thm-KS}(\cite{KunikawaSakurai2022})
Let $(M,g)$ be an $n$-dimensional, complete Riemmanian manifold with compact boundary. For $K\geq 0$, we assume that $Ric_M\geq -(n-1)K$ and $H_{\partial M}\geq -(n-1)\sqrt{K}$. Let $u:B_R(\partial M)\to (0,\infty)$ be a positive harmonic function  with Dirichlet boundary condition. We assume that its derivative $u_{\nu}$ in the direction of the outward unit normal vector $\nu$ is non-negative over $\partial M$. Then we have
\begin{align*}
    \frac{|\nabla u|}{u}\leq C_n\left(\frac{1}{R}+\sqrt{K}\right),
\end{align*}
on $B_{\frac{R}{2}}(\partial M)$, where $C_n$ is a positive constant depending only on $n$, and $B_R(\partial M):=\{x\in M|d(x,\partial M)<R\}$.
\end{theorem}

In particular, they can obtain the following Liouville theorem:

\begin{corollary}\label{cor-KS}
Let $M$ be a complete Riemannian manifold with compact boundary. We assume that $Ric_M\geq 0$ and $H_{\partial M}\geq 0$. Let $u:M\to (0,\infty)$ be a positive harmonic function  with Dirichlet boundary condition. We assume that $u_{\nu}\geq 0$ over $\partial M$. Then $u$ is constant.
\end{corollary}

The purpose of the
present paper is to extend the above estimate from harmonic functions to $p$-harmonic
functions. 
The main estimate and the resulting Liouville theorem are as follows.

\begin{theorem}\label{thm:main}
Let $(M^{n},g)$ be a complete Riemannian manifold with compact boundary. Assume that
for some $K\geq 0$,
\[
\Ric_{M}\,\geq\,-(n-1)K, \qquad H_{\bdry}\,\geq\,-(n-1)\sqrt{K}.
\]
Let $1<p<\infty$, and let $v>0$ be $p$-harmonic on $\tube{R}$ for some $R>0$, constant
on $\bdry$, and satisfying $v_{\outnu}\geq 0$ along $\bdry$. Then there exists a constant
$C_{n,p}>0$, depending only on $n$ and $p$, such that
\begin{equation}\label{eq:main-estimate}
\sup_{\ctube{R/2}}\frac{\lvert\nabla v\rvert}{v}
\,\leq\,
C_{n,p}\Bigl(\frac{1}{R}+\sqrt{K}\Bigr).
\end{equation}
\end{theorem}

\begin{corollary}\label{cor:liouville}
Let $M$ be a complete Riemannian manifold with compact boundary. We assume that $Ric_M\geq 0$ and $H_{\partial M}\geq 0$. Let $v:M\to (0,\infty)$ be a positive $p$-harmonic function  with Dirichlet boundary condition. We assume that $v_{\nu}\geq 0$ over $\partial M$. Then $v$ is
constant on each connected component of $M$.
\end{corollary}

In Kunikawa and Sakurai ~\cite{KunikawaSakurai2022}, the elliptic estimate is obtained from the maximum of
\(F:=(R^2-\rho_{\partial M}^2)\phi\) with \(\phi:=\lvert\nabla u\rvert/u\).
The Laplacian of this product only sees
\(\Delta(\rho_{\partial M}^2)=2+2\rho_{\partial M}\,\Delta\rho_{\partial M}\)
and \(\lvert\nabla\rho_{\partial M}^2\rvert^2=4\rho_{\partial M}^2\),
so Laplacian comparison of \(\rho_{\partial M}\) is enough.
For \(p\neq 2\) the same radial factor, inserted into the linearized
operator \(A=I+(p-2)\frac{\nabla v\otimes\nabla v}{\lvert\nabla v\rvert^2}\),
produces the extra term \(\operatorname{Hess}\rho(\nabla v,\nabla v)\),
which is not controlled by Ricci curvature assumption. On the other hand, Wang and Zhang's interior iteration differentiates the distance only once
and avoids \(\operatorname{Hess}\rho\), but a cutoff in the coordinate
\(\rho\) near \(\partial M\) produces the same term (\cite{WangZhang2011}). The method available
for \(p=2\) near the boundary and the method available for \(p\neq 2\)
in the interior therefore cannot be combined directly by a routine adaptation.

The argument below is designed to circumvent these obstructions.
Instead of cutting off the energy, a radial barrier that is
$p$-subharmonic in the weak sense is constructed near the
boundary;
comparison then yields a bound on the normal derivative, using the distance only
through its gradient and its Laplacian. At an interior maximum of the energy density,
an intrinsic Bochner computation gives a pointwise bound from the Ricci lower bound
alone, with no cutoff. Away from the boundary, the estimate of Wang and Zhang applies
on geodesic balls that do not meet the boundary. Because the maximum principle already
supplies a pointwise bound, Moser iteration is unnecessary.

The subsequent sections are organized as follows: 
Section~\ref{sec:prelim} records the comparison theorems and regularity results used
later. Section~\ref{sec:lemmas} proves the auxiliary bounds, and Section~\ref{sec:proofs}
completes the argument.

\vspace{.2in}

\section{Preliminaries}\label{sec:prelim}

\vspace{.1in}

In this section, we will collect some fundamental materials that will be used in the proof of our main theorem.

Throughout the paper, $(M^{n},g)$ denotes a smooth, complete, connected Riemannian $n$-manifold
with smooth compact boundary $\bdry$, $n\geq 2$. The Riemannian distance is written by $d$,
and the distance to the boundary is
\[
\rho(x)\,:=\,d\bigl(x,\bdry\bigr).
\]
For $R>0$ write
\[
\tube{R}\,:=\,\bigl\{x\in M:\rho(x)<R\bigr\}
\]
for the open $R$-neighborhood of the boundary, and $\ctube{R}$ for its
closure in $M$. Completeness of $M$ and compactness of $\bdry$ imply that
$\ctube{R}$ is compact. The cut locus of the boundary is denoted
$\Cut(\bdry)$; the function $\rho$ is smooth on $\Mreg$,
and $\lvert\nabla\rho\rvert=1$ there. Let $\innerN$ denote the inward unit normal field
along $\bdry$ and $\outnu:=-\innerN$ the outward unit normal. Along \(\partial M\) one has \(\nabla\rho=N\). 

The mean curvature $H$ of $\bdry$ is the trace of the shape operator associated with the
inward unit normal $\innerN$, following the convention of Kunikawa and Sakurai
\cite{KunikawaSakurai2022}. Equivalently, on $\bdry$,
\begin{equation}\label{eq:H-vs-Delta-rho}
H\,=\,-\Delta\rho.
\end{equation}
Let $H_{\bdry}$ denote the infimum of $H$ on $\bdry$. In particular, the lower bound
$H_{\bdry}\geq -(n-1)\sqrt{K}$ is equivalent to $\Delta\rho\leq(n-1)\sqrt{K}$ along
$\bdry$.

For $1<p<\infty$, a function $v$ of class $W^{1,p}_{\mathrm{loc}}$ is said to be
\emph{$p$-harmonic} on an open set $\Omega\subset M$ if it is a weak solution of the
$p$-Laplace equation
\[
\pLap v\,:=\,\dv\bigl(\lvert\nabla v\rvert^{p-2}\nabla v\bigr)\,=\,0
\quad\text{in }\Omega,
\]
that is,
\[
\int_{\Omega}\bigl\langle\lvert\nabla v\rvert^{p-2}\nabla v,\,\nabla\eta\bigr\rangle\,d\mathrm{vol}
\,=\,0
\qquad\text{for every }\eta\in C^{\infty}_{c}(\Omega).
\]
The function $v$ is called \emph{$p$-subharmonic} if $\pLap v\geq 0$ in the weak sense, i.e., if
the same integral is nonpositive for every nonnegative test function. A function $v$
satisfies the Dirichlet condition on $\bdry$ if $v$ is constant along $\bdry$. Partial
derivatives in the inward and outward normal directions are written $v_{\rho}$ and
$v_{\outnu}$ respectively, so that $v_{\outnu}=-v_{\rho}$ along $\bdry$ whenever the
derivatives exist.

\vspace{.1in}

The following is Kasue's Laplacian comparison theorem for the distance to the boundary. For
real numbers $\kappa$ and $\Lambda$, let $s_{\kappa,\Lambda}$ denote the unique solution
of the Jacobi equation $\varphi''+\kappa\varphi=0$ with $\varphi(0)=1$ and
$\varphi'(0)=-\Lambda$.

\begin{theorem}[\cite{Kasue1982}; cf.\ also \cite{KunikawaSakurai2022}]\label{thm:kasue}
Let $\kappa,\Lambda\in\mathbb{R}$. Assume $\Ric_{M}\geq(n-1)\kappa$ and
$H_{\bdry}\geq(n-1)\Lambda$. Then
\[
\Delta\rho\,\leq\,(n-1)\,\frac{s'_{\kappa,\Lambda}(\rho)}{s_{\kappa,\Lambda}(\rho)}
\]
pointwise on $\Mreg$, and the same inequality holds
on $\tube{R}$ in the sense of distributions.
\end{theorem}

The distributional statement is recorded in~\cite[p.~119]{Kasue1983},
following~\cite[Corollary~2.44]{Kasue1982}. Specializing to a Ricci lower bound of the form
\(-(n-1)K\) yields the comparison used throughout the paper. 

\begin{corollary}\label{cor:lap-comp-tube}
Let $K\geq 0$ and assume
\[
\Ric_{M}\,\geq\,-(n-1)K, \qquad H_{\bdry}\,\geq\,-(n-1)\sqrt{K}.
\]
Set $\mu:=(n-1)\sqrt{K}$. Then $\Delta\rho\leq\mu$ pointwise on
$\Mreg$, and $\Delta\rho\leq\mu$ holds on
$\tube{R}$ in the sense of distributions. In particular, $\Delta\rho\leq\mu$ along
$\bdry$.
\end{corollary}

\begin{proof}
Apply Theorem~\ref{thm:kasue} with $\kappa=-K$ and $\Lambda=-\sqrt{K}$. For $K>0$ one
has $s_{-K,-\sqrt{K}}(t)=e^{\sqrt{K}\,t}$, hence
$s'/s=\sqrt{K}$. For $K=0$ the model solution is $s_{0,0}(t)=1$ and $s'/s=0$. The
boundary inequality is~\eqref{eq:H-vs-Delta-rho}.
\end{proof}

\vspace{.1in}

The following local gradient estimate for $p$-harmonic functions due to Wang and Zhang will be applied on geodesic balls that do not meet
the boundary.

\begin{theorem}[\cite{WangZhang2011}]\label{thm:WZ}
Let $(M^{n},g)$ be a complete Riemannian manifold with $\Ric_{M}\geq-(n-1)K$ for some
$K\geq 0$. Let $1<p<\infty$, and let $v>0$ be $p$-harmonic on a geodesic ball $B(o,r)$.
Then there exists a constant $C_{n,p}>0$, depending only on $n$ and $p$, such that
\[
\sup_{B(o,r/2)}\frac{\lvert\nabla v\rvert}{v}
\,\leq\,
C_{n,p}\,\frac{1+\sqrt{K}\,r}{r}.
\]
\end{theorem}

Unlike the estimate of Kotschwar and Ni \cite{KotschwarNi2009}, which requires a lower
bound on sectional curvature, Theorem~\ref{thm:WZ} uses only a Ricci bound. The proof in
\cite{WangZhang2011} proceeds by Moser iteration and differentiates the distance function
only once.

\vspace{.1in}

The weak comparison principle for the $p$-Laplace operator on bounded
domains will also be used; see Lindqvist \cite{Lindqvist2006}.

\begin{proposition}[Weak comparison principle]\label{prop:weak-comparison}
Let $\Omega\subset M$ be a bounded open set, $1<p<\infty$, and let
$a,v\in W^{1,p}(\Omega)\cap C(\overline{\Omega})$. Suppose
\[
-\pLap a\,\leq\,-\pLap v
\quad\text{in the weak sense on }\Omega,
\]
that is,
\[
\int_{\Omega}
\bigl\langle
\lvert\nabla a\rvert^{p-2}\nabla a
-
\lvert\nabla v\rvert^{p-2}\nabla v,\,
\nabla\eta
\bigr\rangle\,d\mathrm{vol}
\,\leq\,0
\]
for every nonnegative $\eta\in C^{\infty}_{c}(\Omega)$, and suppose $a\leq v$ on
$\partial\Omega$. Then $a\leq v$ on $\overline{\Omega}$.
\end{proposition}

\begin{proof}
The argument is standard. Testing with $\eta=(a-v)_{+}$ and using the monotonicity
inequality
\[
\bigl\langle\lvert\xi\rvert^{p-2}\xi-\lvert\zeta\rvert^{p-2}\zeta,\,\xi-\zeta\bigr\rangle
\,\geq\,0
\qquad\text{for all }\xi,\zeta\in T_{x}M
\]
yields $\nabla(a-v)_{+}=0$, hence $(a-v)_{+}\equiv 0$.
\end{proof}

Interior and boundary regularity for the $p$-Laplace equation will be used to justify
pointwise computations and to read off normal derivatives along $\bdry$.

\begin{proposition}[\cite{Tolksdorf1984}, \cite{Lieberman1988}]\label{prop:regularity}
Let $1<p<\infty$ and let $v$ be $p$-harmonic on an open set $\Omega\subset M$.
\begin{enumerate}
\item[\textup{(i)}]
On every compact subset of $\Omega$, the function $v$ is of class $C^{1,\alpha}$ for some
$\alpha\in(0,1)$ depending on $n$, $p$, and the local geometry
\cite{Tolksdorf1984}. On the open set $\{\nabla v\neq 0\}$ the equation is uniformly
elliptic, and $v$ is smooth there by elliptic bootstrapping.
\item[\textup{(ii)}]
If $\Omega$ meets $\bdry$ and $v$ is constant on $\bdry\cap\overline{\Omega}$, then $v$
is of class $C^{1,\alpha}$ up to $\bdry\cap\overline{\Omega}$
\cite{Lieberman1988}. In particular, the normal derivative $v_{\rho}$ exists and is
continuous along $\bdry$.
\end{enumerate}
\end{proposition}

\vspace{.2in}

\section{Auxiliary estimates}\label{sec:lemmas}

Let $v$ be as in Theorem~\ref{thm:main}. Set
\begin{equation}\label{eq:u-f}
u\,:=\,-(p-1)\log v, \qquad f\,:=\,\lvert\nabla u\rvert^{2},
\end{equation}
so that
\[
\frac{\lvert\nabla v\rvert}{v}\,=\,\frac{\sqrt{f}}{p-1}.
\]
The function $u$ is a weak solution of the normalized equation associated with the
$p$-Laplacian. Three estimates for $v$ and $f$ are proved below, corresponding to the
boundary itself, an interior critical point of $f$, and the region at a definite
distance from the boundary.

\subsection{A boundary estimate via a radial barrier}

\begin{lemma}\label{lem:boundary}
Along $\bdry$,
\begin{equation}\label{eq:boundary-bound}
\frac{\lvert\nabla v\rvert}{v}
\,\leq\,
\frac{2}{R}+\frac{n-1}{p-1}\sqrt{K}.
\end{equation}
\end{lemma}

\begin{proof}
Write $v_{0}:=v|_{\bdry}>0$. The assumptions $v_{\outnu}\geq 0$ and $\outnu=-\innerN$ give
\begin{equation}\label{eq:v-rho-sign}
v_{\rho}\,=\,-v_{\outnu}\,\leq\,0
\qquad\text{on }\bdry.
\end{equation}

\medskip
\noindent\textit{Step 1: construction of a radial barrier.}
Let $\delta:=R/2$ and $\Sigma_{\delta}:=\{\rho=\delta\}$. The set $\Sigma_{\delta}$ is a
closed subset of the compact tube $\{\rho\leq\delta\}$, hence compact. Since $v$ is
continuous and positive there,
\[
m_{0}\,:=\,\min_{\Sigma_{\delta}}v\,>\,0.
\]
(If $\Sigma_{\delta}$ is empty, replace $m_{0}$ by $v_{0}$ in what follows; the
construction is unchanged.) Set
\[
\widetilde{m}\,:=\,\min\{m_{0},v_{0}\}, \qquad \mu\,:=\,(n-1)\sqrt{K}.
\]
Let $\psi\in C^{2}([0,\delta])$ satisfy
\begin{equation}\label{eq:barrier-ODE}
(p-1)\psi''+\mu\psi'\,=\,0,
\qquad
\psi(0)=v_{0},
\qquad
\psi(\delta)=\widetilde{m}.
\end{equation}
If $\mu>0$, the equation reads $\psi''/\psi'=-\mu/(p-1)$. Integrating once gives
$\psi'(\rho)=-\alpha\,e^{-\mu\rho/(p-1)}$ for a constant $\alpha$, and a further
integration together with $\psi(0)=v_{0}$ yields
\[
\psi(\rho)
\,=\,
v_{0}-\frac{\alpha(p-1)}{\mu}\Bigl(1-e^{-\mu\rho/(p-1)}\Bigr).
\]
The condition $\psi(\delta)=\widetilde{m}$ determines
\[
\alpha
\,=\,
\frac{(v_{0}-\widetilde{m})\mu}{(p-1)\bigl(1-e^{-\mu\delta/(p-1)}\bigr)}
\,\geq\,0,
\]
where the sign follows from $\widetilde{m}\leq v_{0}$. If $\mu=0$, then $\psi''=0$, so
$\psi(\rho)=v_{0}-\alpha\rho$ with $\alpha=(v_{0}-\widetilde{m})/\delta\geq 0$. In both
cases $\psi'\leq 0$.

\medskip
\noindent\textit{Step 2: the composition $\psi\circ\rho$ is $p$-subharmonic.}
Wherever $\rho$ is smooth, the identity $\lvert\nabla\rho\rvert=1$ gives
\[
\pLap(\psi\circ\rho)
\,=\,
\dv\bigl(\lvert\psi'\rvert^{p-2}\psi'\,\nabla\rho\bigr)
\,=\,
\bigl(\lvert\psi'\rvert^{p-2}\psi'\bigr)'
+
\lvert\psi'\rvert^{p-2}\psi'\,\Delta\rho.
\]
Set $b(\rho):=\lvert\psi'\rvert^{p-2}\psi'$. Then
$b(\rho)=-\alpha^{p-1}e^{-\mu\rho}$ (and $b\equiv-\alpha^{p-1}$ if $\mu=0$), hence
\[
b'(\rho)\,=\,\mu\alpha^{p-1}e^{-\mu\rho}\,=\,\mu\lvert b(\rho)\rvert.
\]
Consequently, at points where $\rho$ is smooth,
\begin{equation}\label{eq:p-subharmonic-smooth}
\pLap(\psi\circ\rho)
\,=\,
\mu\lvert b\rvert+b\,\Delta\rho
\,=\,
\lvert b\rvert\bigl(\mu-\Delta\rho\bigr)
\,\geq\,0,
\end{equation}
where the last inequality uses the pointwise part of Corollary~\ref{cor:lap-comp-tube}
(and holds equally when $\mu=0$). 

Across the cut locus one still has
\(\nabla(\psi\circ\rho)=\psi'\nabla\rho\) almost everywhere. Thus for every
\(\eta\in C_c^\infty(\Omega_\delta)\) with \(\eta\ge 0\),
\[
\int_{\Omega_\delta}
\bigl\langle|\nabla(\psi\circ\rho)|^{p-2}\nabla(\psi\circ\rho),\nabla\eta\bigr\rangle
=\int_{\Omega_\delta}b(\rho)\,\langle\nabla\rho,\nabla\eta\rangle,
\]
where \(\Omega_\delta:=\{0<\rho<\delta\}\). Let \(\varphi:=-b(\rho)\eta\ge 0\).
Approximating \(\varphi\) by nonnegative smooth test functions and using
\(\Delta\rho\le\mu\) in the distributional sense ,
one has \(-\int\langle\nabla\rho,\nabla\varphi\rangle\le\mu\int\varphi\).
Since \(\nabla\varphi=-b'(\rho)\eta\nabla\rho-b(\rho)\nabla\eta\) and
\(|\nabla\rho|=1\) a.e., this rearranges to
\[
\int_{\Omega_\delta}b(\rho)\,\langle\nabla\rho,\nabla\eta\rangle
\le\int_{\Omega_\delta}\bigl(\mu|b|-b'\bigr)\eta=0.
\]
Hence \(\Delta_p(\psi\circ\rho)\ge 0\) on \(\Omega_\delta\) in the weak sense.

\medskip
\noindent\textit{Step 3: application of the weak comparison principle.}
The domain $\Omega_{\delta}$ is bounded. The function $a:=\psi\circ\rho$ is Lipschitz,
since $\rho$ is $1$-Lipschitz and $\psi$ is smooth, while $v$ is continuous on the
compact set $\overline{\Omega_{\delta}}$ and of class $C^{1,\alpha}$ in the interior by
Proposition~\ref{prop:regularity}\textup{(i)}. Thus
$a,v\in W^{1,p}(\Omega_{\delta})\cap C(\overline{\Omega_{\delta}})$. Step~2 gives
$-\pLap a\leq 0=-\pLap v$. The boundary of $\Omega_{\delta}$ consists of $\bdry$ and
$\Sigma_{\delta}$: on $\bdry$ one has $a=\psi(0)=v_{0}=v$, while on $\Sigma_{\delta}$ one
has $a=\psi(\delta)=\widetilde{m}\leq m_{0}\leq v$. Proposition~\ref{prop:weak-comparison}
therefore yields
\begin{equation}\label{eq:barrier-below}
\psi\circ\rho\,\leq\,v
\qquad\text{on }\overline{\Omega_{\delta}}.
\end{equation}

\medskip
\noindent\textit{Step 4: reading the normal derivative.}
Since $\bdry$ is smooth and the boundary data are constant, Proposition~\ref{prop:regularity}
ensures that $v$ is of class $C^{1}$ up to $\bdry$, so $v_{\rho}$ exists along $\bdry$.
For $x\in\bdry$, let $\gamma(t)=\exp_{x}(t\innerN)$ be the inward unit-speed geodesic.
For all sufficiently small $t>0$ one has $\rho(\gamma(t))=t$. Inequality
\eqref{eq:barrier-below} then gives
\[
\frac{v(\gamma(t))-v_{0}}{t}
\,\geq\,
\frac{\psi(t)-\psi(0)}{t},
\]
and sending $t\downarrow 0$ gives $v_{\rho}(x)\geq\psi'(0)=-\alpha$. Combined with
\eqref{eq:v-rho-sign}, this yields $0\geq v_{\rho}\geq-\alpha$. Since $v$ is constant on
$\bdry$, its tangential gradient vanishes, and therefore
\[
\lvert\nabla v\rvert\,=\,\lvert v_{\rho}\rvert\,\leq\,\alpha
\qquad\text{on }\bdry.
\]

\medskip
\noindent\textit{Step 5: the numerical bound.}
If $\mu>0$, set $\tau:=\mu\delta/(p-1)>0$. Then
\begin{align*}
\frac{\alpha}{v_{0}}
&\,=\,
\Bigl(1-\frac{\widetilde{m}}{v_{0}}\Bigr)
\frac{\mu}{(p-1)\bigl(1-e^{-\tau}\bigr)}
\,\leq\,
\frac{\mu}{(p-1)\bigl(1-e^{-\tau}\bigr)}
\,=\,
\frac{1}{\delta}\cdot\frac{\tau}{1-e^{-\tau}}.
\end{align*}
The elementary inequality $e^{\tau}\geq 1+\tau$ for $\tau>0$ rearranges to
$(1+\tau)(1-e^{-\tau})\geq\tau$, or equivalently
$\tau/(1-e^{-\tau})\leq 1+\tau$. Hence
\[
\frac{\alpha}{v_{0}}
\,\leq\,
\frac{1+\tau}{\delta}
\,=\,
\frac{1}{\delta}+\frac{\mu}{p-1}
\,=\,
\frac{2}{R}+\frac{n-1}{p-1}\sqrt{K}.
\]
If $\mu=0$, then $\alpha/v_{0}\leq 1/\delta=2/R$, and the same formula holds. This proves
\eqref{eq:boundary-bound}.
\end{proof}

\subsection{An interior maximum principle for the energy density}

The following lemma is a self-contained Bochner computation at a critical point of $f$.

\begin{lemma}\label{lem:critical}
Let $x_{0}$ be an interior point of $M$ at which $f(x_{0})>0$ and $f$ attains a local
maximum. Then
\[
f(x_{0})\,\leq\,(n-1)^{2}K.
\]
\end{lemma}

\begin{proof}
Since $f(x_{0})>0$, one has $\nabla u\neq 0$ near $x_{0}$, so the $p$-Laplace equation
is uniformly elliptic in a neighborhood of $x_{0}$. Proposition~\ref{prop:regularity}\textup{(i)}
together with Schauder bootstrapping implies that $u$ is smooth near $x_{0}$, and the
computations below are pointwise.

\medskip
\noindent\textit{Step 1: the equation satisfied by $u$.}
From $\nabla u=-(p-1)\nabla v/v$ it follows that
\[
\lvert\nabla u\rvert^{p-2}\nabla u
\,=\,
-(p-1)^{p-1}\frac{\lvert\nabla v\rvert^{p-2}\nabla v}{v^{p-1}}.
\]
Taking the divergence and using $\dv(\lvert\nabla v\rvert^{p-2}\nabla v)=0$ yields
\begin{align*}
\dv\bigl(f^{p/2-1}\nabla u\bigr)
&\,=\,
-(p-1)^{p-1}
\bigl\langle\lvert\nabla v\rvert^{p-2}\nabla v,\,\nabla(v^{-(p-1)})\bigr\rangle
\\
&\,=\,
(p-1)^{p}\frac{\lvert\nabla v\rvert^{p}}{v^{p}}
\,=\,
f^{p/2}.
\end{align*}
Expanding the left-hand side yields, on the set $\{f>0\}$,
\begin{equation}\label{eq:Delta-u}
\Delta u
\,=\,
f-\Bigl(\frac{p}{2}-1\Bigr)f^{-1}\langle\nabla f,\nabla u\rangle.
\end{equation}
The Bochner formula for $u$ reads
\[
\frac12\Delta f
\,=\,
\lvert\nablasq u\rvert^{2}+\Ric(\nabla u,\nabla u)+\langle\nabla\Delta u,\nabla u\rangle.
\]

\medskip
\noindent\textit{Step 2: simplifications at the critical point.} 
At $x_{0}$ choose an orthonormal frame with $e_{1}=\nabla u/\sqrt{f}$, so that
$u_{j}=\sqrt{f}\,\delta_{1j}$. Since $\nabla f(x_{0})=0$,
\[
0\,=\,f_{i}\,=\,2\sum_{j}u_{ji}u_{j}\,=\,2\sqrt{f}\,u_{1i}
\qquad(i=1,\dots,n),
\]
hence $u_{1i}=0$ for all $i$, and in particular $u_{11}=0$. Equation~\eqref{eq:Delta-u}
at $x_{0}$, where $\langle\nabla f,\nabla u\rangle=0$, reduces to $\Delta u=f$, and
therefore
\[
\sum_{i=2}^{n}u_{ii}\,=\,\Delta u-u_{11}\,=\,f.
\]
Cauchy--Schwarz then gives
\begin{equation}\label{eq:hessian-CS}
\lvert\nablasq u\rvert^{2}
\,\geq\,
\sum_{i=2}^{n}u_{ii}^{2}
\,\geq\,
\frac{1}{n-1}\Biggl(\sum_{i=2}^{n}u_{ii}\Biggr)^{2}
\,=\,
\frac{f^{2}}{n-1}
\qquad\text{at }x_{0}.
\end{equation}

\medskip
\noindent\textit{Step 3: cancellation of the Hessian of $f$.}
Differentiating~\eqref{eq:Delta-u} and pairing with $\nabla u$ at $x_{0}$ (several terms
drop because $\nabla f=0$) gives
\[
\langle\nabla\Delta u,\nabla u\rangle
\,=\,
-\Bigl(\frac{p}{2}-1\Bigr)f^{-1}\bigl\langle\nabla\langle\nabla f,\nabla u\rangle,\nabla u\bigr\rangle
\,=\,
-\Bigl(\frac{p}{2}-1\Bigr)f^{-1}\nablasq f(\nabla u,\nabla u).
\]
The last step uses the product rule
\[
\nabla u\bigl\langle\nabla f,\nabla u\bigr\rangle
\,=\,
\nablasq f(\nabla u,\nabla u)+\bigl\langle\nabla f,\,\nabla_{\nabla u}\nabla u\bigr\rangle,
\]
whose second summand vanishes at $x_{0}$. Substituting into the Bochner formula gives
at $x_{0}$
\[
\Delta f
\,=\,
2\lvert\nablasq u\rvert^{2}+2\Ric(\nabla u,\nabla u)
-(p-2)f^{-1}\nablasq f(\nabla u,\nabla u).
\]

\medskip
\noindent\textit{Step 4: the linearized operator.}
On $\{f>0\}$ define the bundle endomorphism
\[
A\,:=\,I+(p-2)\frac{\nabla u\otimes\nabla u}{f}
\]
and the operator
\[
L\zeta
\,:=\,
\dv\bigl(f^{p/2-1}A\nabla\zeta\bigr)
-p\,f^{p/2-1}\langle\nabla u,\nabla\zeta\rangle.
\]
At $x_{0}$ the drift term vanishes, and
\[
L(f)
\,=\,
f^{p/2-1}\,A:\nablasq f
\,=\,
f^{p/2-1}
\Bigl(
\Delta f+(p-2)f^{-1}\nablasq f(\nabla u,\nabla u)
\Bigr).
\]
Inserting the expression for $\Delta f$ from Step~4, the two terms involving
$\nablasq f(\nabla u,\nabla u)$ cancel, and
\[
L(f)(x_{0})
\,=\,
2f^{p/2-1}\bigl(\lvert\nablasq u\rvert^{2}+\Ric(\nabla u,\nabla u)\bigr).
\]
Using~\eqref{eq:hessian-CS} and $\Ric(\nabla u,\nabla u)\geq-(n-1)Kf$, one finds
\[
L(f)(x_{0})
\,\geq\,
2f^{p/2-1}\Biggl(\frac{f^{2}}{n-1}-(n-1)Kf\Biggr).
\]

\medskip
\noindent\textit{Step 5: the maximum principle.}
Since $f$ attains a local maximum at $x_{0}$, one has $\nablasq f(x_{0})\leq 0$. The
endomorphism $A$ is symmetric and positive definite (its eigenvalues are $p-1>0$ in the
direction of $\nabla u$ and $1$ in the orthogonal complement), hence
\[
L(f)(x_{0})\,=\,f^{p/2-1}A_{ij}f_{ij}\,\leq\,0.
\]
Combining the two inequalities for $L(f)(x_{0})$ and dividing by $2f^{p/2}(x_{0})>0$
gives
\[
\frac{f}{n-1}\,\leq\,(n-1)K,
\]
that is, $f(x_{0})\leq(n-1)^{2}K$.
\end{proof}

\subsection{An interior estimate away from the boundary}

\begin{lemma}\label{lem:interior}
Let $x\in\tube{R}$ satisfy $R/4\leq\rho(x)\leq R/2$. Then
\[
\frac{\lvert\nabla v\rvert}{v}(x)
\,\leq\,
C_{n,p}\Bigl(\frac{1}{R}+\sqrt{K}\Bigr),
\]
for some constant $C_{n,p}$ depending only on $n$ and $p$.
\end{lemma}

\begin{proof}
Apply Theorem~\ref{thm:WZ} on the geodesic ball $B(x,R/8)$. For every
$y\in B(x,R/8)$, the triangle inequality and the fact that $\rho$ is $1$-Lipschitz give
\begin{align*}
\rho(y)&\,\geq\,\rho(x)-d(x,y)\,>\,\frac{R}{4}-\frac{R}{8}\,=\,\frac{R}{8}\,>\,0,\\
\rho(y)&\,\leq\,\rho(x)+d(x,y)\,<\,\frac{R}{2}+\frac{R}{8}\,=\,\frac{5R}{8}\,<\,R.
\end{align*}
Hence
\[
B(x,R/8)\,\subset\,\bigl\{R/8<\rho<5R/8\bigr\}\,\subset\,\tube{R}\setminus\bdry.
\]
In particular the ball does not meet $\bdry$. The function $v$ is positive and
$p$-harmonic on $B(x,R/8)$, and $\Ric_{M}\geq-(n-1)K$ holds on $M$. Theorem~\ref{thm:WZ}
therefore yields
\begin{align*}
\frac{\lvert\nabla v\rvert}{v}(x)
&\,\leq\,
\sup_{B(x,R/16)}\frac{\lvert\nabla v\rvert}{v}
\,\leq\,
C_{n,p}\,\frac{1+\sqrt{K}\,(R/8)}{R/8}
\\
&\,=\,
C_{n,p}\Bigl(\frac{8}{R}+\sqrt{K}\Bigr)
\,\leq\,
8C_{n,p}\Bigl(\frac{1}{R}+\sqrt{K}\Bigr),
\end{align*}
which is the claimed bound upon redefining the dimensional constant.
\end{proof}

\vspace{.2in}

\section{Proofs of the main results}\label{sec:proofs}

\vspace{.1in}

\begin{proof}[Proof of Theorem~\ref{thm:main}]
The closed tube $\ctube{R/2}$ is a closed and bounded subset of a complete
manifold, lying at distance at most $R/2$ from the compact set $\bdry$, hence is itself
compact. On this set $v$ is continuous and positive (of class $C^{1,\alpha}$ in the
interior and of class $C^{1}$ up to $\bdry$, by Proposition~\ref{prop:regularity}), so
$v\geq c>0$ and the energy density
\[
f\,=\,(p-1)^{2}\frac{\lvert\nabla v\rvert^{2}}{v^{2}}
\]
is continuous on $\ctube{R/2}$. It therefore attains its maximum at some point
$x_{0}$. There are three cases.

\medskip
\noindent\textit{Case 1:} $\rho(x_{0})=0$, that is, $x_{0}\in\bdry$.
Lemma~\ref{lem:boundary} gives
\[
\sqrt{f}(x_{0})
\,=\,
(p-1)\frac{\lvert\nabla v\rvert}{v}(x_{0})
\,\leq\,
\frac{2(p-1)}{R}+(n-1)\sqrt{K}.
\]

\medskip
\noindent\textit{Case 2:} $0<\rho(x_{0})<R/2$.
In this case $x_{0}$ has a neighborhood contained in $\ctube{R/2}$, so $f$
attains a local maximum at $x_{0}$. If $f(x_{0})>0$, Lemma~\ref{lem:critical} yields
\[
\sqrt{f}(x_{0})\,\leq\,(n-1)\sqrt{K}.
\]
If $f(x_{0})=0$, then $f$ vanishes identically on the tube and there is nothing to
estimate.

\medskip
\noindent\textit{Case 3:} $\rho(x_{0})=R/2$.
Here $x_{0}$ is a boundary point of the region $\ctube{R/2}$, so $f$ need not
attain a local maximum in $M$ and Lemma~\ref{lem:critical} does not apply. On the other
hand $x_{0}$ is an interior point of $M$ with $\rho(x_{0})=R/2\geq R/4$, and
Lemma~\ref{lem:interior} gives
\[
\frac{\lvert\nabla v\rvert}{v}(x_{0})
\,\leq\,
C_{n,p}\Bigl(\frac{1}{R}+\sqrt{K}\Bigr).
\]

\medskip
Combining the three cases gives
\begin{align*}
\sup_{\ctube{R/2}}\sqrt{f}
&\,\leq\,
\max\Biggl\{
\frac{2(p-1)}{R}+(n-1)\sqrt{K},\;
(n-1)\sqrt{K},\;
C_{n,p}\Bigl(\frac{1}{R}+\sqrt{K}\Bigr)
\Biggr\}
\\
&\,\leq\,
C'_{n,p}\Bigl(\frac{1}{R}+\sqrt{K}\Bigr).
\end{align*}
Dividing by $p-1$ yields~\eqref{eq:main-estimate}.
\end{proof}

\begin{proof}[Proof of Corollary~\ref{cor:liouville}]
Here $K=0$. Suppose first that $x$ lies in a connected component of $M$ that meets
$\bdry$. The function $v$ is $p$-harmonic on all of $M$, so for every $R>2\rho(x)$
Theorem~\ref{thm:main} applies on $\tube{R}$. Since $x\in\ctube{R/2}$,
\[
\frac{\lvert\nabla v\rvert}{v}(x)\,\leq\,\frac{C_{n,p}}{R}.
\]
Sending $R\to\infty$ yields $\nabla v(x)=0$. If $x$ lies in a connected component that
does not meet $\bdry$, apply Theorem~\ref{thm:WZ} with $K=0$ on geodesic balls $B(x,R)$
in that component:
\[
\sup_{B(x,R/2)}\frac{\lvert\nabla v\rvert}{v}\,\leq\,\frac{C_{n,p}}{R}.
\]
Letting $R\to\infty$ again gives $\nabla v(x)=0$. Thus $\nabla v\equiv 0$ on $M$, and
$v$ is constant on each connected component.
\end{proof}

\section*{Acknowledgements}
The author thanks his advisor Professor Jun Sun for the suggestion of the problem, guidance during this work, and comments on the introduction.


\begin{thebibliography}{11}


\bibitem{ChengYau1975}
S.~Y. Cheng and S.~T. Yau,
\textit{Differential equations on Riemannian manifolds and their geometric applications},
Comm. Pure Appl. Math. \textbf{28} (1975), 333--354.

\bibitem{Kasue1982}
A.~Kasue,
\textit{A Laplacian comparison theorem and function theoretic properties of a complete Riemannian manifold},
Japan. J. Math. (N.S.) \textbf{8} (1982), 309--341.

\bibitem{Kasue1983}
A.~Kasue,
\textit{Ricci curvature, geodesics and some geometric properties of
Riemannian manifolds with boundary},
J.\ Math.\ Soc.\ Japan \textbf{35} (1983), 117--131.

\bibitem{KotschwarNi2009}
B.~Kotschwar and L.~Ni,
\textit{Local gradient estimates of $p$-harmonic functions, $1/H$-flow, and an entropy formula},
Ann. Sci. \'Ecole Norm. Sup. (4) \textbf{42} (2009), 1--36.

\bibitem{KunikawaSakurai2022}
K.~Kunikawa and Y.~Sakurai,
\textit{Yau and Souplet--Zhang type gradient estimates on Riemannian manifolds with boundary under Dirichlet boundary condition},
Proc. Amer. Math. Soc. \textbf{150} (2022), 1767--1777.


\bibitem{Lieberman1988}
G.~M. Lieberman,
\textit{Boundary regularity for solutions of degenerate elliptic equations},
Nonlinear Anal. \textbf{12} (1988), 1203--1219.


\bibitem{Lindqvist2006}
P.~Lindqvist,
\textit{Notes on the $p$-Laplace equation},
Report~102, University of Jyv\"askyl\"a, Jyv\"askyl\"a, 2006.

\bibitem{Tolksdorf1984}
P.~Tolksdorf,
\textit{Regularity for a more general class of quasilinear elliptic equations},
J. Differential Equations \textbf{51} (1984), 126--150.


\bibitem{WangZhang2011}
X.~Wang and L.~Zhang,
\textit{Local gradient estimate for $p$-harmonic functions on Riemannian manifolds},
Comm. Anal. Geom. \textbf{19} (2011), 759--771.

\bibitem{Yau1975}
S.~T. Yau,
\textit{Harmonic functions on complete Riemannian manifolds},
Comm. Pure Appl. Math. \textbf{28} (1975), 201--228.
\end{thebibliography}
\end{document}